\documentclass[preprint]{article}

 \PassOptionsToPackage{round}{natbib}

\usepackage{nips_2018}

\usepackage[utf8]{inputenc} 
\usepackage[T1]{fontenc}    
\usepackage{url}            
\usepackage{booktabs}       
\usepackage{amsfonts}       
\usepackage{nicefrac}       
\usepackage{microtype}      

\usepackage{paralist, amsmath, amsthm, amssymb, color, graphicx, hyperref}
\DeclareGraphicsExtensions{.jpg,.pdf,.eps}
\usepackage{amssymb}
\usepackage{setspace}
\usepackage{mathtools}
\usepackage{color}
\usepackage{algorithm,algorithmic}
\usepackage{float}
\usepackage{caption}

\usepackage{paralist}

\usepackage{stmaryrd}

\usepackage{xfrac}

\usepackage{mathrsfs}
\usepackage{amscd}
\usepackage{dsfont}
\usepackage{tikz}

\usepackage{enumerate}

\theoremstyle{plain}
\newtheorem{thm}{Theorem}[section] 
\newtheorem{lem}[thm]{Lemma}
\newtheorem{prop}[thm]{Proposition}
\newtheorem{cor}[thm]{Corollary}

\theoremstyle{definition}
\newtheorem{defn}[thm]{Definition} 

\newtheorem{remark}[thm]{Remark}

\usepackage{listings}
\usepackage{color} 
\definecolor{mygreen}{RGB}{28,172,0} 
\definecolor{mylilas}{RGB}{170,55,241}
\definecolor{mygray}{gray}{0.95}

\newcommand{\mscr}[1]{\mathscr{#1}}

\newcommand{\ZZ}{\mathbb{Z}}
\newcommand{\RR}{\mathbb{R}}
\newcommand{\NN}{\mathbb{N}}

\newcommand{\EE}{\mathbb{E}}
\newcommand{\PP}{\mathbb{P}}

\newcommand{\al}{\alpha}

\newcommand{\ga}{\gamma}
\newcommand{\de}{\delta}
\newcommand{\ep}{\epsilon}

\newcommand{\ta}{\theta}
\newcommand{\Ta}{\Theta}

\newcommand{\sa}{\sigma}

\usepackage{mathtools}

\renewcommand{\l}{\left}
\renewcommand{\r}{\right}

\newcommand{\defeq}{\vcentcolon=}

 \newcommand{\iid}{\overset{\text{iid}}{\sim}}

\newcommand{\ul}{\underline}
\newcommand{\ds}{\displaystyle}

\title{Differentially Private Uniformly Most Powerful Tests for Binomial Data}

\author{
  Jordan Awan\\
  Department of Statistics\\
  Penn State University\\
  University Park, PA 16802 \\
  \texttt{awan@psu.edu} 
  \And
  Aleksandra Slavkovi\'c\\
  Department of Statistics\\
  Penn State University\\
  University Park, PA 16802 \\
  \texttt{sesa@psu.edu} \\ 
}
\date{}

\begin{document}

\maketitle

\begin{abstract}
  We derive uniformly most powerful (UMP)  tests for simple and one-sided hypotheses for a population proportion within the framework of Differential Privacy (DP), optimizing finite sample performance. We show that in general,  DP hypothesis tests for exchangeable data can always be expressed as a function of the empirical distribution. Using this structure, we prove a `Neyman-Pearson lemma' for binomial data under DP, where the DP-UMP only depends on the sample sum. Our tests can also be stated as a post-processing of a random variable, whose distribution we coin ``Truncated-Uniform-Laplace'' (Tulap), a generalization of the Staircase and discrete Laplace distributions. 
Furthermore, we obtain exact $p$-values, which are easily computed in terms of the Tulap random variable.
  We show that our results also apply to distribution-free hypothesis tests for continuous data. Our simulation results demonstrate that our tests have exact type I error, and are more powerful than current techniques.
\end{abstract}

\section{Introduction}\label{Introduction}

Differential Privacy (DP), introduced by \citet{Dwork2006:Sensitivity}, offers a rigorous measure of disclosure risk. To satisfy DP, a procedure cannot be a deterministic function of the sensitive data, but must incorporate additional randomness, beyond sampling. Subject to the DP constraint, it is natural to search for a procedure which maximizes the utility of the output. Many works address the goal of minimizing the distance between the output of the  randomized DP procedure and standard non-private algorithms, but few attempt to  infer properties about the underlying population (with a few notable exceptions, see related work), which is typically the goal in statistics and scientific research. 
In this paper, we study the setting where each individual contributes a sensitive binary value, and we wish to infer the population proportion via hypothesis tests, subject to DP. In particular, we derive {\em uniformly most powerful} (UMP) tests for simple and one-sided hypotheses, optimizing finite sample performance.

UMP tests are fundamental to classical statistics, being closely linked to  sufficiency, likelihood inference, and optimal confidence sets. However, finding UMP tests can be hard and in many cases they do not even exist (see \citet[Section 4.4]{Schervish1996}). Our results are the first to achieve  UMP tests under $(\epsilon,\delta)-$DP, and may be the first step towards a general theory of optimal inference under DP.

{\bfseries \large  Related work} \citet{Vu2009} were among the first to perform hypothesis tests under DP. They developed private tests for population proportions as well as for independence in $2\times 2$ contingency tables. In both settings, they fix the noise adding distribution, and use approximate sampling distributions to perform these DP tests. \citet{Gaboardi2016} and \citet{Wang2015:Revisiting} build on \citet{Vu2009}, developing additional tests for multinomial data. \citet{Wang2015:Revisiting} develop asymptotic sampling distributions for their tests, verifying via simulations that the type I errors are reliable. \citet{Gaboardi2016} on the other hand, use simulations to compute an empirical type I error. \citet{Uhler2013} develop DP chi-squared tests and $p$-values for GWAS data, and derive the exact sampling distribution of their noisy statistic. \citet{Gaboardi2018} work under ``Local Differential Privacy'', a stronger notion of privacy than DP, and develop multinomial tests based on asymptotic distributions. Given a DP output, \citet{Sheffet2017} and \citet{Barrientos2017} develop significance tests for regression coefficients. 


Outside the hypothesis testing setting, there has been some work on optimal population inference under DP. \citet{Duchi2017} give general techniques to derive minimax rates under local DP, and in particular give minimax optimal point estimates for the mean, median, generalized linear models, and nonparametric density estimation. \citet{Karwa2017} develop nearly optimal confidence intervals for normally distributed data with finite sample guarantees, which could potentially be inverted to give UMP-unbiased tests.

Related work on developing optimal DP mechanisms for general loss functions such as \citet{Geng2016} and \citet{Ghosh2009}, give mechanisms that optimize  symmetric convex loss functions, centered at a real statistic.
\citet{Awan2018:Structure} derive the optimal mechanism in a similar sense, among the class of $K$-Norm Mechanisms.

{\bfseries \large Our contributions}
The previous literature on DP hypothesis testing has a few characteristics in common: 1) nearly all of these works first add noise to the data, and perform their test as a post-processing procedure, 2) all of the hypothesis tests use either asymptotic distributions or simulations to derive approximate decision rules, and 3) while each procedure is derived intuitively based on classical theory, none show that they are optimal among all possible DP algorithms.

In contrast, in this paper we search over all DP hypothesis tests at level $\al$, deriving the \emph{uniformly most powerful} (UMP) test. 
As in \citet{Vu2009}, we focus on the problem of testing a population proportion under DP. However, rather than fixing the noise distribution, we search over all  DP  tests at level $\al$ to find the UMP test. In Theorem \ref{SuffStatThm}, we show that for exchangeable data, DP tests need only depend on the empirical distribution. We use this structure to find UMP tests for simple hypotheses in Theorems
\ref{UMP1} and \ref{UMP2}, and extend these results to obtain one-sided UMP tests in Corollary \ref{OneSide1}. These  tests are closely tied to our proposed ``\emph{Truncated-Uniform-Laplace}'' (\emph{Tulap}) distribution, which  extends both the discrete Laplace distribution (studied in \citet{Ghosh2009}), and the Staircase distribution of \citet{Geng2016} to the setting of $(\ep,\de)$-DP. We prove that the Tulap distribution satisfies $(\ep,\de)$-DP in Theorem \ref{TulapDP}. We show that our UMP tests can be stated as a post-processing of a Tulap random variable, from which we obtain exact $p$-values via Theorem \ref{PValueResult} and Algorithm \ref{PValueAlgorithm}. 
In Section \ref{DistributionFree}, we show that our results apply to distribution-free hypothesis tests of continuous data. In Section \ref{Simulations}, we verify through simulations that our UMP tests have exact type I error, and are more powerful than current techniques.

\section{Background and notation}\label{Background}
Let $\mscr X$ be any set. The $n$-fold cartesian product of $\mscr X$ is $\mscr X^n = \{(X_1,X_2,\ldots, X_n) \mid x_i \in \mscr X\}$. The \emph{Hamming distance} metric on $\mscr X^n$ is $\de: \mscr X^n \times \mscr X^n \rightarrow \ZZ^{\geq 0}$, defined by $\de(X,X') =\# \{i \mid X_i\neq X'_i\}$. 

Differential Privacy, introduced by \citet{Dwork2006:Sensitivity}, provides a formal measure of disclosure risk. The notion of DP that we give in Definition \ref{DPDefn} more closely resembles the formulation given by \citet{Wasserman2010:StatisticalFDP} which uses the language of distributions rather than random mechanisms.
\begin{defn}
  [Differential Privacy: \citet{Dwork2006:Sensitivity,Wasserman2010:StatisticalFDP}]\label{DPDefn}
Let $\ep>0$, $\de\geq 0$, and $n\in \{1,2,\ldots\}$ be given. Let $\mscr X$ be any set, and $(\mscr Y, \mscr F)$ be a measurable set. Let $\mscr P = \{P_x\mid x\in \mscr X^n\}$ be a set of probability measures on $(\mscr Y, \mscr F)$. We say that $\mscr P$ satisfies \emph{$(\ep, \de)$-Differential Privacy} ($(\ep,\de)$ - DP) if for all $B\in \mscr F$ and all $X,X' \in \mscr X^n$ such that $\de(X,X')=1$, we have $P_X(B) \leq e^\ep P_{X'}(B) + \de.$
\end{defn}
In Definition \ref{DPDefn}, $\mscr X$ is the set of possible values that one individual can contribute to the database. Intuitively, if a set of distributions satisfies $(\ep,\de)$-DP for small values of $\ep$ and $\de$, then if one person's data is changed in the database, the distribution does not change much. We call $\ep$ the privacy budget, ideally a small value less than $1$, and $\de$ typically $\ll \frac 1n$ allows us to disregard events which have small probability. We refer to $(\ep, 0)$-DP as pure DP, and $(\ep, \de)$-DP as approximate DP.

Proposition \ref{PostProcessing} states that if the random variables $\{Z| X\mid X\in \mscr X^n\}$ satisfy $(\ep,\de)$-DP, then applying any function to $Z| X$  also satisfies $(\ep,\de)$-DP; see \citet{Dwork2014:AFD} for a proof.

\begin{prop}[Post-processing]\label{PostProcessing}
Let $\ep>0$, $\de\geq 0$, $\mscr X$ be any set, and $\mscr P = \{P_x \mid x\in \mscr X^n\}$ be a set of probability measures on $(\mscr Y, \mscr F)$ which satisfies $(\ep, \de)$-DP. Let $f: \mscr Y \rightarrow \mscr Z$ be some measurable function. Then $\mscr P\circ f^{-1} = \{P_x\circ f^{-1} \mid x\in \mscr X^n\}$ satisfies $(\ep, \de)$-DP.
\end{prop} 

The focus of this paper is to find uniformly most powerful (UMP) hypothesis tests, subject to DP. As the output of a DP method is necessarily a random variable, we work with randomized hypothesis tests, which we review in Definition \ref{HT}. Our notation follows that of \citet[Chapter 4]{Schervish1996}.

\begin{defn}
  [Hypothesis Test]\label{HT}
  Let $(X_1,\ldots, X_n)\in \mscr X^n$ be distributed $X_i \iid P_\ta$
  , where $\ta\in \Ta$. Let $\Ta_0, \Ta_1$ be a partition of $\Ta$. A \emph{(randomized) test} of $H_0: \ta \in \Ta_0$ versus $H_1: \ta\in \Ta_1$ is a measurable function $\phi: \mscr X^n \rightarrow [0,1]$. We say a test $\phi$ is at \emph{level} $\al$ if $\sup_{\ta\in \Ta_0} \EE_{P_\ta} \phi \leq \al$, and at \emph{size} $\al$ if $\sup_{\ta\in \Ta_0} \EE_{P_\ta} \phi = \al$. The \emph{power} of $\phi$ at $\ta$ is denoted $\beta_\phi(\ta) = \EE_{P_\ta}\phi$. 

Let $\Phi$ be a set of tests for $H_0: \ta \in \Ta_0$ versus $H_1:\ta\in \Ta_1$. We say that $\phi^*\in \Phi$ is the \emph{uniformly most powerful} (UMP) test among $\Phi$ at level $\al$ if $\sup_{\ta\in \Ta_0}\beta_{\phi^*}(\ta)\leq \al$ and for any $\phi \in \Phi$ such that $\sup_{\ta\in \Ta_0} \beta_{\phi}(\ta)\leq \al$ we have 
$\beta_{\phi^*}(\ta) \geq \beta_{\phi}(\ta)$, for all $\ta\in \Ta_1$.
\end{defn}

Typically, we use capital letters to denote random variables and lowercase letters for particular values. For a random variable $X$, we denote $F_X$ as its cumulative distribution function (cdf), and either $f_X$ as its density or $p_X$ as its point mass function (pmf), depending on whether $X$ is continuous or discrete.

Finally, we use the \emph{nearest integer function} $[\cdot]:\RR\rightarrow \ZZ$ extensively. For any real number $t\in \RR$, $[t]$ is defined to be the integer nearest to $a$. To ensure that this function is well defined, we take $[z+1/2]$ to be the nearest even integer when $z\in \ZZ$. Given this definition, $[-t]=t$ for all $t\in \RR$.

\section{Exchangeability condition}\label{Setup}
Let $(\mscr X, \mscr B)$ be a measurable space, and $\ul X \in \mscr X^n$. Assume that $\ul X$ has an \emph{exchangeable} distribution $\PP$. 
Denote by $\PP_n$ the empirical distribution of $\ul X$, defined by $\PP_n(B) = \frac 1n \sum_{i=1}^n I_B(X_i)$ for all $B\in \mscr X$. Note that $\PP_n$ determines $\ul X$ up to permutations. By De Finetti's Theorem (see \citet[Theorem 1.48]{Schervish1996}), $\PP_n$ is a sufficient statistic for $\PP$, and from classical statistical theory (see \citet[Section 2.1.2]{Schervish1996}, we know that any hypothesis test about $\PP$ need only depend a sufficient statistic.

In Theorem \ref{SuffStatThm}, we show that the same result holds for hypothesis tests under DP as well. Let $\phi: \mscr X^n \rightarrow [0,1]$ be any test. Interpreting $\phi_{\ul X}$ as the probability of `Reject', the test $\phi$ satisfies $(\ep, \de)$-DP if and only if for all $\ul X,\ul X'\in \{0,1\}^n$ such that $\de(\ul X,\ul X')=1$,
\begin{equation}
  \label{DPBinary}
\phi_{\ul X} \leq e^\ep \phi_{\ul X'} + \de\quad \text{and} \quad (1- \phi_{\ul X}) \leq e^\ep (1-\phi_{\ul X'}) + \de.
  \end{equation}
  \begin{thm}
    \label{SuffStatThm}
Let $\{\mu_i\}_{i\in I}$ be a set of  exchangeable  distributions on $(\mscr X^n,\mscr B)$. Let $\phi: \mscr X^n \rightarrow [0,1]$ satisfy \eqref{DPBinary}. Define $\phi': \mscr X^n \rightarrow [0,1]$ by $\phi'_{\ul X} = \frac{1}{n!} \sum_{\pi \in \sa(n)} \phi_{\pi(\ul X)}$, where $\sa(n)$ is the symmetric group on $n$ letters. Then $\phi'$ satisfies \eqref{DPBinary}, $ \int \phi'_{\ul X} \ d\mu_i = \int \phi_{\ul X} \ d\mu_i$  $\forall i\in I$, and $\phi'$ only depends on $\PP_n$.
\end{thm}
\begin{proof}[Proof.]
For any $\pi \in \sa(n)$, $\phi_{\pi(\ul X)}$ satisfies $(\ep,\de)$-DP. By exchangeability, $\int \phi_{\pi(\ul X)}\ d\mu_i = \int \phi_{\ul X}\ d\mu_i$. Since condition \ref{DPBinary} is closed under convex combinations, and integrals are linear, the result follows.
\end{proof}
Theorem \ref{SuffStatThm} says that for any DP hypothesis tests, we can construct another test with the same power at every $\mu_i$, which only depends on $\PP_n$.

The particular problem we study is as follows. Let $\ul X \in \{0,1\}^n$, where  $X_i$ is the sensitive data of individual $i$, and assume that $\ul X$ is exchangeable. Then by De Finetti's Theorem (see \citet[Theorem 1.47]{Schervish1996}), there exists $\ta$ such that $X_i | \ta \iid \mathrm{Bern}(\ta)$, and hence $X|\ta \sim \mathrm{Binom}(n,\ta)$, where $X = \sum \ul X$. By Theorem \ref{SuffStatThm}, we can restrict our attention to tests which are functions of $X$. A test $\phi:\{0,1,\ldots, n\}\rightarrow[0,1]$ satisfies $(\ep,\de)$ -DP  if and only if for all $x\in \{1,2,3,\ldots, n\}$,
\begin{align}
  \phi_x &\leq e^\ep \phi_{x+1}+\de\label{DP1}\\
\phi_{x+1}&\leq e^\ep \phi_{x}+\de\label{DP2}\\
(1-\phi_x)&\leq e^\ep(1-\phi_{x+1})+\de\label{DP3}\\
(1-\phi_{x+1})&\leq e^\ep(1-\phi_x)+\de\label{DP4}.
\end{align}
We denote the set of all tests which satisfy \eqref{DP1}-\eqref{DP4} as  
$\mscr D^n_{\ep,\de} =\big\{\phi:\phi \text{ satisfies \eqref{DP1}-\eqref{DP4}}\big\}.$

\section{The Tulap distribution}
All of our results are related to the proposed \emph{Truncated-Uniform-Laplace} (Tulap) distribution, defined in Definition \ref{TulapDefn}, which we motivate as follows. \citet{Geng2016} show that for general loss functions, adding $L\sim \mathrm{DLap}(e^{-\ep})$ to $X$ is optimal  under $(\ep,0)$-DP. We suspect that post-processing $X+L$ may give optimal inference for $(\ep,0)$-DP. However, we know by classical UMP theory that a randomized test is required, since $X+L$ is discrete. Alternatively, we add a random variable $U\sim \mathrm{Unif}(-1/2,1/2)$ to obtain a continuous distribution. We call the distribution of $(X+L+U)\mid {X}$ as $\mathrm{Tulap}(X,b,0)$. $\mathrm{Tulap}(X,b,q)$ is obtained by truncating within the central $(1-q)^{th}$-quantiles of $\mathrm{Tulap}(X,b,0)$.

\begin{defn}\label{TulapDefn}
  Let $N$ and $N_0$ be real-valued random variables. Let $m\in \RR$, $b\in(0,1)$ and $q\in [0,1)$. We say that $N_0 \sim \mathrm{Tulap}(m,b,0)$ and $N\sim \mathrm{Tulap}(m,b,q)$ if $N_0$ and $N$ have the following cdfs:
\[F_{N_0}(x) = \begin{cases}
\frac{b^{-[x-m]}}{1+b} \l(b+(x-m-[x-m]+\frac 12)(1-b)\r)&x\leq [m]\\
1-\frac{b^{[x-m]}}{1+b}\l(b+ ([x-m]-(x-m)+\frac 12 ) (1-b)\r) &x>[m],
\end{cases}\]
and $F_N(x) = \l(\frac{F_{N_0}(x) - \frac q2}{1-q}\r)I\{\frac q2\leq F_{N_0}(x) \leq 1-\frac q2\} + I\{F_{N_0}(x)>1-\frac q2\}$.
\end{defn}
While the cdf in Definition \ref{TulapDefn} is complicated, Algorithm \ref{SampleTulap} gives a simple method of sampling a random Tulap variable, which we show to be correct in Lemma \ref{TulapLem}.
\begin{algorithm}
\caption{Sample from Tulap distribution}
\scriptsize
INPUT: $m\in \RR$, $b\in (0,1)$, $q\in [0,1)$.
\begin{algorithmic}[1]
  \setlength\itemsep{0em}
  \STATE Draw $G_1,G_2 \iid \mathrm{Geom}(1-b)$ and $U\sim \mathrm{Unif}(-1/2,1/2)$
\STATE Set $N=G_1-G_2+U+m$
\STATE If $F_{N_0}(N)<q/2$ or $F_{N_0}(N)>1-q/2$, where $N_0\sim \mathrm{Tulap}(m,b,0)$,
go to 1:
\end{algorithmic}
OUTPUT: $N$
\label{SampleTulap}
\end{algorithm}
\begin{lem}\label{TulapLem}
  \begin{enumerate}
  \item Let $L\sim \mathrm{DLap}(b)$, $U\sim \mathrm{Unif}(-1/2,1/2)$, $G_1,G_2\iid \mathrm{Geom}(1-b)$, and $N_0\sim \mathrm{Tulap}(m,b,0)$, where the pmf of $L$ is $p_L(x) = \frac{1-b}{1+b}b^{|x|}$ for $x\in \ZZ$, and the pmf of $G_1$ is $p_{G_1}(x)=(1-p)^xp$ for $x\in \{0,1,2,\ldots\}$. Then $L+U+m\overset d= G_1-G_2+U+m\overset d= N_0$.
\item Let $N$ be the output of Algorithm \ref{SampleTulap} with inputs $m,b,q$. Then $N\sim \mathrm{Tulap}(m,b,q)$
\item The random variable $N\sim \mathrm{Tulap}(m,b,q)$ is continuous and symmetric about $m$.
  \end{enumerate}
\end{lem}
\begin{proof}
  It is well known that $L\overset d= G_1-G_2$ (see \citet{Inusah2006}). To derive $F_{N_0}$, we compute the convolution of $L$ and $U$ to get  $f_{L+U}(x) = f_U(x-[x]) p_L([x])$. Integrating gives $F_{L+U}$.
 If $q=0$, it is clear that Algorithm \ref{SampleTulap} is correct. If $q>0$, then by rejection sampling, we have that $N\sim \mathrm{Tulap}(m,b,q)$ (see \citet[Chapter 11]{Bishop2006} for an introduction to rejection sampling).
Property 3 is clear by inspection.
\end{proof}

In Section \ref{pValues}, we find that our UMP tests can be achieved by post-processing the random variable $Z|X\sim \mathrm{Tulap}(X,b=e^{-\ep},q=\frac{2\de b}{1-b-2\de b})$. Theorem  \ref{TulapDP} tells us that we can release $Z| X$ at $(\ep,\de)$-DP, and all of our UMP tests and $p$-values can be computed by post-processing $Z$.

\begin{thm}\label{TulapDP}
Let $\mscr X$ be any set, and $T:\mscr X^n\rightarrow \ZZ$, with $\Delta(T) = \sup |T(X)-T(X')|=1$, where the supremum is over the set $\{(X,X')\in \mscr X^n\times \mscr X^n\mid \de(X,X')=1\}$. Then the set of distributions $\l\{\mathrm{Tulap}\l(T(X),b=e^{-\ep}, \frac{2\de b}{1-b+2\de b}\r)\middle| X\in\mscr X^n\r\}$ satisfies $(\ep,\de)$-DP.
\end{thm}

The proof of Theorem  \ref{TulapDP} follows from Lemmas \ref{RecurrenceLem1}, \ref{RecurrenceLem2}, and \ref{CDFLem}, and  is deferred to Section \ref{pValues}.

\begin{remark}
  $\mathrm{Tulap}(0,b,0) \overset d = \mathrm{Staircase}(b,1/2)$ and $[\mathrm{Tulap}(0,b,0)] \overset d= \mathrm{DLap}(b)$, where $\mathrm{Staircase}(b,\ga)$ is the distribution in \citet{Geng2016}. 
\citet{Geng2016} show that for real valued statistics $T$ and convex symmetric loss functions centered at $T$, the optimal noise distribution for $\ep$-DP is $\mathrm{Staircase}(b,\ga)$ for $b=e^{-\ep}$ and some $\ga\in (0,1)$. If the statistic is a count, then \citet{Ghosh2009} show that $\mathrm{DLap}(b)$ is optimal. Our results agree with these works when $\de=0$, and extend them to the case of arbitrary $\de$.
\end{remark}

\section{UMP  tests when $\de=0$}\label{SectionDz}

The goal of this section is as follows: given $n, \ep>0, \al>0, \ta_0<\ta_1$, and $X\sim Binom(n,\ta)$, find the UMP test at level $\al$ among $\mscr D^n_{\ep,0}$ for testing simple hypotheses $H_0:\ta = \ta_0$ versus $H_1: \ta=\ta_1$.

In the classic statistic setting, the UMP for this test is given by the \emph{Neyman-Pearson lemma}, however in the DP framework, our test must satisfy \eqref{DP1}-\eqref{DP4}. Within these constraints, we follow the logic behind the Neyman-Pearson lemma as follows.
Let $\phi_x\in \mscr D_{\ep,0}^n$. Thinking of $\phi_x$ defined recursively, equations \eqref{DP1}-\eqref{DP4} give  upper bounds and  lower bounds for $\phi_x$ in terms of $\phi_{x-1}$. Since $\ta_1>\ta_0$, and binomial distributions have a monotone likelihood ratio (MLR) in $X$,  larger values of $X$ give more evidence for $\ta_1$ over $\ta_0$. Thus,  we expect $\phi_x$ is increasing as much as possible, subject to  \eqref{DP1}-\eqref{DP4}. Lemma \ref{RecurrenceLem1} shows that taking $\phi_x$ to be such a function is equivalent to having $\phi_x$ be the cdf of a Tulap random variable. The properties in Lemma \ref{CLem} are easily verified, so the proof is omitted.

\begin{lem}\label{CLem}
Let $N\sim \mathrm{Tulap}(m,b,q)$ and let $t\in \ZZ$. Then 
$\ds F_{N}(t)= \begin{cases}
b^{-t}C(m)&t\leq [m]\\
1-b^tC(-m)&t>[m],
\end{cases}$,
where $C(m) = (1+b)^{-1} b^{[m]} (b+([m]-m+1/2)(1-b))$.
It is true that $C(m)$ is positive, monotone decreasing, and continuous in $m$. Furthermore, $b^{-[m]}C(m) = 1-b^{[m]}C(-m)$.
\end{lem}

\begin{lem}\label{RecurrenceLem1}
  Let $\ep>0$ be given. Let $\phi:\{0,1,2,\ldots, n\} \rightarrow (0,1)$. The following are equivalent:
  \begin{enumerate}
  \item There exists $m\in (0,1)$ such that $\phi_0=m$ and $\phi_x = \min\{e^\ep \phi_{x-1} , 1-e^{-\ep} (1-\phi_{x-1})\}$ for $x=1,\ldots, n$.
\item There exists $m\in (0,1)$ such that $\phi_0=m$ and for $x=1,\ldots, n$,\\
$\phi_x = \begin{cases}
e^\ep \phi_{x-1} & \phi_{x-1} \leq \frac{1}{1+e^\ep}\\
1-e^{-\ep}(1-\phi_{x-1})&\phi_{x-1}>\frac{1}{1+e^\ep}.
\end{cases}$
\item There exists $m\in \RR$ such that $\phi_x = F_{N_0}(x-m)$ for $x=0,1,2,\ldots, n$, where $N_0\sim \mathrm{Tulap}(0,b=e^{-\ep},0)$.
  \end{enumerate}
\end{lem}
\begin{proof}[Proof Sketch.]
First show that 1 and 2 are equivalent by checking which constraint is active. Then verify that $F_{N_0}(x-m)$ satisfies the recurrence of 2, using the properties in Lemma \ref{CLem}.
  \end{proof}

  It remains to show that the tests in Lemma \ref{RecurrenceLem1} are in fact UMP. The main tool used to prove this is Lemma \ref{BingLem}, which is a standard result in the classical hypothesis testing theory.

\begin{lem}
\label{BingLem}
  Let $(\mscr X, \mscr F,\mu)$ be a measure space and let $f,g$ be two densities on $\mscr X$ with respect to $\mu$. Suppose that $\phi_1,\phi_2: \mscr X\rightarrow [0,1]$ are both measurable functions $\mscr R/\mscr F$ such that $\int \phi_1 f \ d\mu \geq \int \phi_2 f \ d\mu$, and there exists $k\geq 0$ such that $\phi_1\geq \phi_2$ when $g\geq kf$ and $\phi_1\leq \phi_2$ when $g< k f$. Then $\int \phi_1 g \ d\mu \geq \int \phi_2 g \ d\mu$.
\end{lem}
\begin{proof}
  Note that $(\phi_1 - \phi_2)(g-kf)\geq 0$ for almost all $x\in \mscr X$ (with respect to $\mu$). This implies that $\int (\phi_1 - \phi_2)(g-kf) \ d\mu \geq 0$. Hence, $\int \phi_1 g \ d\mu - \int \phi_2 g \ d\mu \geq k\l(\int \phi_1 f\ d\mu - \int \phi_2 f \ d\mu\r) \geq 0$.
\end{proof}

Next we present our key result, Theorem \ref{UMP1}, which can be viewed as a `Neyman-Pearson lemma' for binomial data under $(\ep,0)$-DP. 

\begin{thm}\label{UMP1}
  Let $\ep>0,\al\in (0,1),0<\ta_0<\ta_1<1$, and $n\geq1$ be given. Observe $X\sim \mathrm{Binom}(n,\ta)$, where $\ta$ is unknown. Let $N_0 \sim \mathrm{Tulap}(0,b=e^{-\ep},0)$. Set the decision rule $\phi^*: \ZZ\rightarrow [0,1]$ by $\phi^*_x = F_{N_0}(x-m)$, where $m$ is chosen such that $E_{\ta_0}\phi^*_x=\al$. Then $\phi^*$ is UMP-$\al$ test of $H_0: \ta=\ta_0$ versus $H_1: \ta = \ta_1$ among $\mscr D_{\ep,0}^n$.
\end{thm}
\begin{proof}[Proof Sketch.]
  Let $\phi$ be any other test which satisfies \eqref{DP1}-\eqref{DP4} at level $\al$. Then, since $\phi^*$ can be written in the form of 1 in Lemma \ref{RecurrenceLem1}, there exists $y\in \ZZ$ such that $\phi^*_x\geq \phi_x$ when $x\geq y$ and $\phi^*_x\leq \phi_x$ when $x<y$. By MLR of the binomial distribution and Lemma \ref{BingLem}, we have $\beta_{\phi^*}(\ta_1)\geq \beta_{\phi}(\ta_1)$.
\end{proof}

While the classical Neyman-Pearson lemma results in an acceptance and rejection region, the DP-UMP always has some probability of rejecting the null, due to the constraints \eqref{DP1}-\eqref{DP4}. As $\ep\uparrow \infty$, the DP-UMP converges to the non-private UMP.

\section{UMP  tests when $\de\geq0$}\label{SectionDnz}
In this section, we extend the results of Section \ref{SectionDz} to allow for $\de\geq 0$. We begin by proposing the form of the UMP-$\al$ test for simple hypotheses. As in Section \ref{SectionDz}, we suspect that the  UMP test is increasing in $x$ as much as \eqref{DP1}-\eqref{DP4} allow. 
Lemma \ref{RecurrenceLem2} states that such a test can be written as the cdf of a Tulap random variable. We omit the proof of Theorem \ref{UMP2}, which mimics the proof of Theorem \ref{UMP1}.

\begin{lem}\label{RecurrenceLem2}
  Let $\ep>0$ and $\de\geq 0$ be given and set $q = \frac{2\de b}{1-b+2\de b}$. Let $\phi:\{0,1,2,\ldots, n\}\rightarrow [0,1]$. The following are equivalent:
  \begin{enumerate}
  \item There exists $y\in \{0,1,2,\ldots, n\}$ and $m \in (0,1)$ such that $\phi_x=0$ for $x<y$, $\phi_y=m$ and for $x>y$,
$\phi_x = 
\min\{e^{\ep} \phi_{x-1}+\de,\quad 1-e^{-\ep}(1-\phi_{x-1})+e^{-\ep}\de,\quad 1\}
$
\item There exists $y\in \{0,1,2,\ldots, n\}$ and $m\in (0,1)$ such that $\phi_x=0$ for $x<y$, $\phi_y=m$, and 
\[\phi_x = \begin{cases}
e^\ep \phi_{x-1}+\de&x>y \text{ and }\phi_{x-1} \leq \frac{1-\de}{1+e^{\ep}}\\
1-e^{-\ep}(1-\phi_{x-1})+e^{-\ep}\de&\frac{1-\de}{1+e^{\ep}} \leq \phi_{x-1} \leq 1-\de\\
1&\phi_{x-1} >1-\de
\end{cases}\]
\item There exists $m\in \RR$ such that 
$\phi_x = F_N(x-m)$ where $N\sim \mathrm{Tulap}(0,e^{-\ep},q)$.
  \end{enumerate}
\end{lem}
\begin{proof}[Proof Sketch.]
 The equivalence of 1 and 2 only requires determining which constraints are active. To show the equivalence of 2 and 3, we verify that $F_{N}(x-m)$ satisfies the recurrence of 2, using the relation $F_N(x) = (1-q)^{-1}(F_{N_0}(x) - \frac q2) I\{\frac q2 \leq F_{N_0}(x)\leq 1-\frac q2\}$ and  Lemma \ref{RecurrenceLem1}.
\end{proof}

\begin{thm}\label{UMP2}
  Let $\ep>0$, $\de\geq 0$, $\al \in (0,1)$, $0<\ta_0<\ta_1<1$, and $n\geq 1$ be given. Observe $X\sim \mathrm{Binom}(n,\ta)$, where $\ta$ is unknown. Set $b=e^{-\ep}$ and $q = \frac{2\de b}{1-b+2\de b}$. Define $\phi^*: \ZZ \rightarrow [0,1]$ by $\phi^*_x = F_N(x-m)$ where $N\sim \mathrm{Tulap}(0,b,q)$ and $m$ is chosen such that $E_{\ta_0}\phi^*_x=\al$. Then $\phi^*$ is UMP-$\al$ test of $H_0: \ta=\ta_0$ versus $H_1: \ta = \ta_1$ among $\mscr D_{\ep,\de}^n$.
\end{thm}

So far we have focused on simple hypothesis tests, but since our test only depends on $\ta_0$, and not on $\ta_1$, our test is in fact UMP for one-sided tests, as stated in Corollary \ref{OneSide1}. 




\begin{cor}\label{OneSide1}
Let $X\sim \mathrm{Binom}(n,\ta)$. Set $\phi^*_x = F_N(x-m_1)$ and $\psi^*_x = 1-F_N(x-m_2)$, where $N\sim \mathrm{Tulap}\l(0, e^{-\ep}\frac{2\de b}{1-b+2\de b}\r)$ and $m_1,m_2$ are chosen such that $E_{\ta_0}\phi^*_x=\al$ and $E_{\ta_0} \psi^*_x=\al$. Then $\phi^*_x$ is UMP-$\al$ among $\mscr D_{\ep,\de}^n$ for testing $H_0: \ta\geq \ta_0$ versus $H_1: \ta<\ta_0$, and $\psi^*_x$ is UMP-$\al$ among $\mscr D_{\ep,\de}^n$ for testing $H_0: \ta\leq \ta_0$ versus $H_1: \ta> \ta_0$.
\end{cor}

\section{Optimal one-sided p-values}\label{pValues}
Inference based on developed  UMP-$\al$ tests for one-sided hypotheses under $(\ep,\de)$-DP is simply to accept or reject $H_0$. In scientific research, however, $p$-values are more preferred way for weighing the evidence in favor of the alternative hypothesis over the null. In this section, we show that these UMP tests can be achieved by post-processing a Tulap random variable, and using this,  we develop an algorithm for releasing a $p$-value which agrees with the UMP tests in Sections \ref{SectionDz} and \ref{SectionDnz}.

Since our UMP test from Theorem \ref{UMP2} rejects  with probability $\phi^*_x = F_N(x-m)$, given $N\sim F_N$, $\phi^*_x$ rejects the null if and only if $X+N\geq m$. If  $\{X+N\mid X\in \{0,1,\ldots, n\}\}$ satisfies $(\ep,\de)$-DP, our test is just a post-processing of $X+N$. To prove Theorem \ref{TulapDP}, we require  Lemma \ref{CDFLem}.

\begin{lem}\label{CDFLem}
  Observe $X\in \mscr X^n$. Let $T: \mscr X^n \rightarrow \RR$, and let $\{\mu_X\mid X\in \mscr X^n\}$ be a set of probability measures on $\RR$, dominated by Lebesgue measure. Suppose that $\mu_X$ is parametrized by $T(X)$ and $\mu_X$ has MLR in $T(X)$. Then $\{\mu_X\}$ satisfies $(\ep,\de)$-DP if and only if for all $\de(X_1,X_2)=1$ and all $t\in \RR$,
  \begin{align}
    \mu_{X_1}((-\infty,t)) &\leq e^{\ep} \mu_{X_2}((-\infty,t))+\de,\\
\mu_{X_1}((t,\infty))&\leq e^{\ep} \mu_{X_2}((t,\infty)) + \de.
  \end{align}
\end{lem}

\begin{proof}[Proof Sketch.]
As in \citet{Wasserman2010:StatisticalFDP} and \citet{Kairouz2017}, we interpret  DP as a constraint on the power of  hypothesis tests for $H_0: X=X_1$ versus $H_1: X=X_2$. Then, for a fixed type I error, the Neyman-Pearson Lemma tells us that the tests with highest power have rejection regions which are half intervals.
  \end{proof}

Note that the Tulap random variables stated in the Theorem \ref{TulapDP} are continuous and have MLR in $T(X)$. By Theorem Lemma \ref{RecurrenceLem2}, we know that the cdfs of these random variables satisfy the conditions of Lemma \ref{CDFLem}. The result of Theorem \ref{TulapDP} follows.

By Theorem \ref{TulapDP}, releasing $X+N$ satisfies $(\ep,\de)$-DP. By Proposition \ref{PostProcessing}, once we release $X+N$, any function of $X+N$ also satisfies $(\ep,\de)$-DP. Thus, we can compute our UMP-$\al$ tests as a function of $X+N$ for any $\al$. By definition, the smallest $\al$ for which we reject the null is the $p$-value for that test. In fact Algorithm \ref{PValueAlgorithm} and Theorem \ref{PValueResult} give a more elegant method of computing this $p$-value.

\begin{thm}\label{PValueResult}
  Let $\ep>0$, $\de\geq 0$, $X\sim \mathrm{Binom}(m,\ta)$ where $\ta$ is unknown, and $Z|X \sim \mathrm{Tulap}(X,b=e^{-\ep}, q= \frac{2\de b}{1-b+2\de b})$. Then 
  \begin{enumerate}
  \item $p(\ta_0,Z) \defeq P(X+N\geq Z\mid Z)$ is a $p$-value for $H_0: \ta\leq \ta_0$ versus $H_1: \ta>\ta_0$, where the probability is over $X\sim \mathrm{Binom}(n,\ta_0)$ and $ N\sim \mathrm{Tulap}(0,b,q)$.
\item Let $0<\al<1$ be given. The test $\phi_X^* = P_{Z| X \sim \mathrm{Tulap}(X,b,q)} (p(\ta_0,Z)\leq \al \mid X)$ is UMP-$\al$ for $H_0: \ta\leq \ta_0$ versus $H_1: \ta>\ta_0$ among $\mscr D_{\ep,\de}^n$.
\item The output of Algorithm \ref{PValueAlgorithm} is equal to $p(\ta_0,Z)$.
  \end{enumerate}
\end{thm}
\begin{algorithm}
\caption{UMP one-sided $p$-value for binomial data under $(\ep,\de)$-DP}
\scriptsize
INPUT: $n\in\NN$, $\ta_0\in (0,1)$, $\ep>0$, $\de\geq 0$,  $Z\sim \mathrm{Tulap}\l(X,b=e^{-\ep},q=\frac{2\de b}{1-b+2\de b}\r)$,
\begin{algorithmic}[1]
  \setlength\itemsep{0em}
  \STATE Set $F_N$ as the cdf of $N\sim \mathrm{Tulap}(0,b,q)$
\STATE Set $\ul F = (F_N(0-Z),F_N(1-Z),\ldots, F_N(n-Z))^\top$
\STATE Set $\ul B = (\binom n0 \ta_0^0(1-\ta_0)^{n-0}, \binom n1 \ta_0^1(1-\ta_0)^{n-1},\ldots, \binom nn \ta_0^n (1-\ta_0)^{n-n})^\top$
\end{algorithmic}
OUTPUT: $\ul F^\top \ul B$
\label{PValueAlgorithm}
\end{algorithm}
 Theorem \ref{PValueResult} tells us that $p(\ta_0,Z)$ is the smallest possible $p$-value for the hypothesis test $H_0: \ta\leq \ta_0$ versus $H_1: \ta>\ta_0$ under $(\ep,\de)$-DP. Note that $1-p(\ta_0,Z) = P(X+N\leq Z\mid Z)$ is the $p$-value for $H_0: \ta\geq \ta_0$ versus $H_1: \ta<\ta_0$ which agrees with the UMP-$\al$ test in Corollary \ref{OneSide1}.

\section{Application to distribution-free inference}\label{DistributionFree}
In this section, we show how our UMP tests for count data can be used to test certain hypotheses for continuous data. In particular, we give a DP version of the sign and median test allowing one to test the median of either paired or independent samples. First we will recall the sign and median tests (see \citet[Sections 5.4 and 6.4]{Gibbons2014} for a more thorough introduction).

{\bfseries Sign test:}
 Suppose we observe $n$ iid pairs $(X_i,Y_i)$ for $i=1,\ldots, n$. Then for all $i=1,\ldots,n$, $X_i \overset d= X$ and $Y_i \overset d =Y$ for some random variables $X$ and $Y$. We assume that for any pair $(X_i,Y_i)$ we can determine if $X_i>Y_i$ or not. For simplicity, we also assume that there are no pairs with $X_i=Y_i$. Denote the unknown probability $p = P(X>Y)$. We want to test a hypothesis such as $H_0: p=.5$ versus $H_1: p>.5$. The sign test uses the test statistic $T = \#\{X_i>Y_i\}$.

{\bfseries Median test:}
Suppose we observe two independent sets of iid data $\{X_i\}_{i=1}^n$ and $\{Y_i\}_{i=1}^n$, where all $X_i$ and $Y_i$ are distinct values, and we have a total ordering on these values. Then there exists random variables $X$ and $Y$ such that  $X_i\overset d =X$ and $Y_i \overset d = Y$ for all $i$. We want to test $H_0: \mathrm{median}(X)\leq \mathrm{median}(Y)$ versus $H_1: \mathrm{median}(X)>\mathrm{median}(Y)$. The median test uses the test statistic $T = \#\{i \mid \mathrm{rank}(X_i) > n\}$, where $\mathrm{rank}(X_i) = \#\{X_j\leq X_i\}+\#\{Y_j\leq X_i\}$.

{\bfseries Privatized tests:} In either setting, under the null hypothesis, the test statistic is distributed as $T \sim \mathrm{Binom}(n,.5)$. In both cases, $T$ satisfies the requirements of the statistic in Theorem \ref{TulapDP}, so $T+N$ satisfies $(\ep,\de)$-DP, where $N\sim \mathrm{Tulap}(0,b,q)$ for $b=e^{-\ep}$ and $q = \frac{2\de b}{1-b-2\de b}$. Using Algorithm \ref{PValueAlgorithm}, we  obtain a private $p$-value for either the sign test or the median test as a post-processing of $T+N$. 

 \section{Simulations}\label{Simulations}
 In this section, we study both the empirical power and the empirical type I error of  our DP-UMP test against the normal approximation proposed by \citet{Vu2009}. For our simulations, we focus on small samples as the noise introduced by DP methods is most impactful in this setting.

In Figure \ref{fig:Power}, we plot the empirical power of our UMP test, the Normal Approximation from \citet{Vu2009}, and the non-private UMP. For each $n$, we generate 10,000 samples from $\mathrm{Binom}(n,.95)$. We  privatize each $X$ by adding $N\sim\mathrm{Tulap}(0,e^{-\ep},0)$ for the DP-UMP and $L\sim \mathrm{Lap}(1/\ep)$ for the Normal Approximation. We compute the UMP $p$-value via Algorithm \ref{PValueAlgorithm} and the approximate $p$-value for $X+L$, using the cdf of $N\l(X, n/4+2/\ep^2\r)$. The empirical power is given by $(10000)^{-1}\#\{\text{$p$-value$<.05$}\}$. The DP-UMP test  indeed gives higher power compared to the Normal Approximation, but the approximation does not lose too much power, however next we see that  type I error is another issue.

In Figure \ref{fig:TypeIerror} we plot the empirical type I error of the DP-UMP and the Normal Approximation tests. We fix $\ep=1$ and $\de=0$, and vary $\ta_0$. For each $\ta_0$, we generate 100,000 samples from $\mathrm{Binom}(30,\ta_0)$. For each sample, we compute the DP-UMP and Normal Approximation tests at size $\al=.05$. We plot the proportion of times we reject the null as well as moving average curves. The DP-UMP, which is provably at size $\al=.05$ achieves type I error very close to $.05$, but the Normal Approximation has a higher type I error for small values of $\ta_0$, and a lower type I error for large values of $\ta_0$. 
\begin{minipage}{.48\linewidth}
  \includegraphics[width = \linewidth]{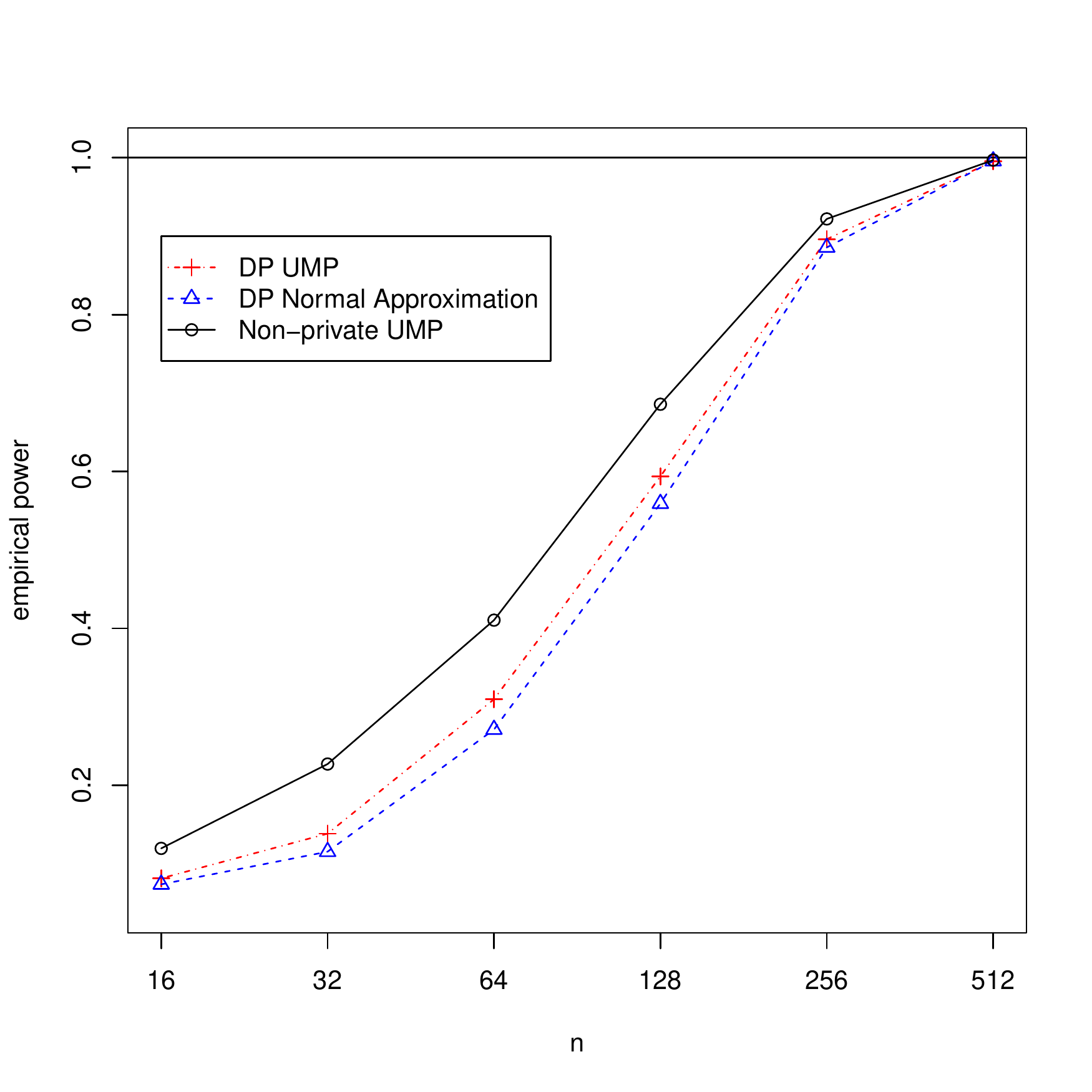}
  \captionof{figure}{Empirical power for UMP and Normal Approximation tests for $H_0: \theta\leq.9$ versus $H_1: \theta\geq .9$. The true value is $\theta = .95$. $\ep=1$ and $\de=0$. $n$ varies along the $x$-axis.}
  \label{fig:Power}
\end{minipage}
\hspace{.02\linewidth}
\begin{minipage}{.48\linewidth}
  \includegraphics[width = \linewidth]{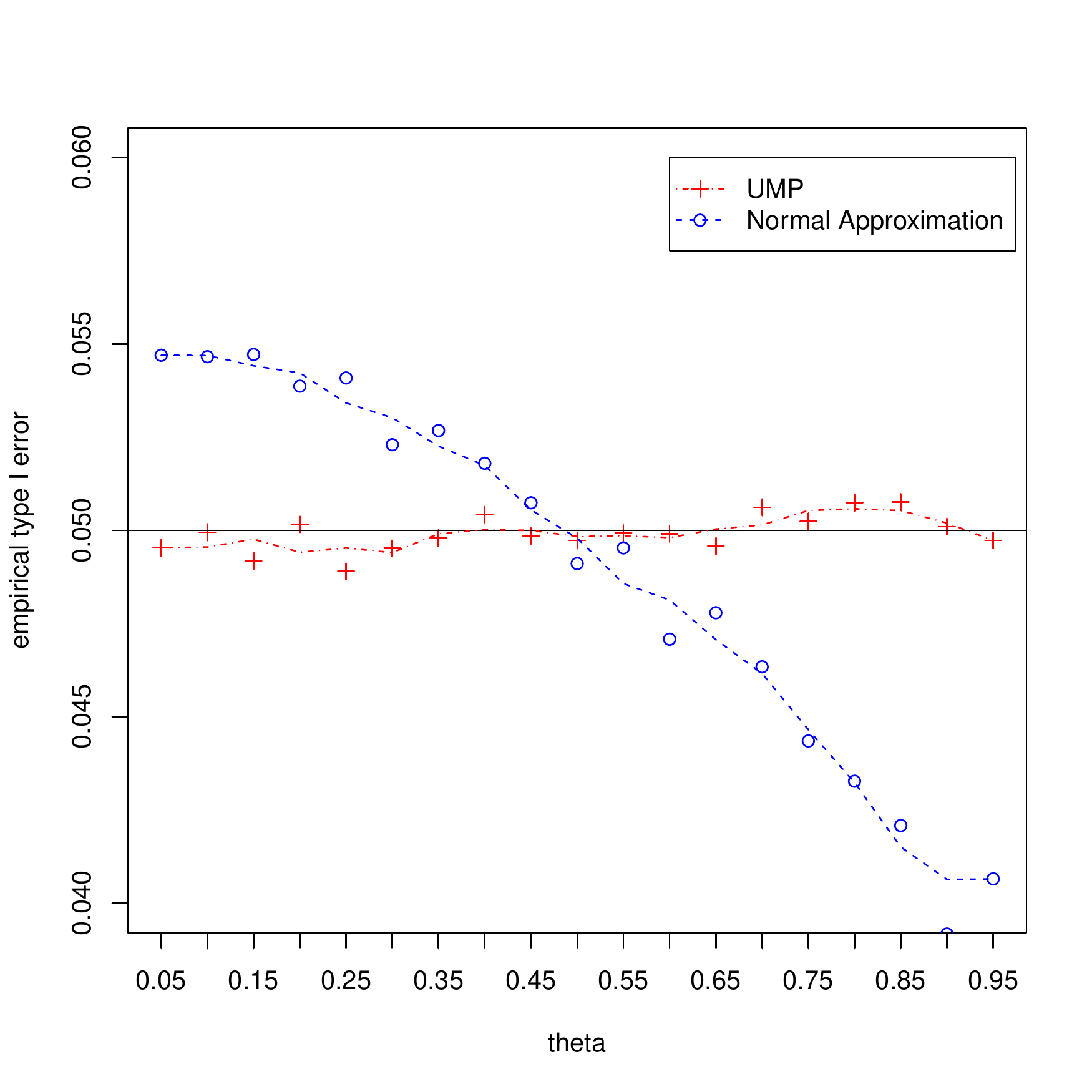}
  \captionof{figure}{Empirical type I error $\al$ for UMP and Normal Approximation tests for $H_0: \theta \leq \ta_0$ versus $H_1: \theta \geq \theta_0$.  $\ta_0$ varies along the $x$-axis. $n=30$, $\ep=1$, and $\de=0$. Target is  $\al=.05$.}
  \label{fig:TypeIerror}
\end{minipage}
 \section{Discussion and future directions}
 In this paper, we  derived the uniformly most powerful test for binary data among all DP $\alpha$-level tests. Previously, while various hypothesis tests under DP have been proposed, none have satisfied such an optimality criterion. Furthermore, since our UMP tests can be achieved by post-processing a noisy sufficient statistic, we are able to produce optimal $p$-values as well. Our results can also be applied to distribution-free tests,  to test some hypotheses about continuous data under DP.

 These results can also be applied to obtain optimal length confidence intervals. In fact, the $p$-value function for the test $H_0: \ta\geq \ta_0$ versus $H_1: \ta<\ta_0$ is a cdf which generates a confidence distribution (See \citet{Xie2013} for a review). Since this $p$-value corresponds to a UMP test,  this confidence distribution is stochastically more concentrated than any other confidence distribution.  In the full paper, we  explore this idea, demonstrating that our results lead to optimal length confidence intervals.

We showed that for exchangeable data, optimal DP tests are functions of the empirical distribution. For binary data, the empirical distribution is equivalent to the sample sum, which is complete sufficient for the binomial model. However, in general it is not clear whether optimal DP tests are always a function of complete sufficient statistics as is the case for classical UMP tests. It would be worth investigating whether there is alternative notion of sufficiency which applies for DP tests.

 When $\de=0$, our optimal noise adding mechanism is related to the discrete Laplace distribution, which \citet{Ghosh2009} and \citet{Geng2016} also found was optimal for a general class of loss functions. For $\de>0$, a truncated discrete Laplace distribution was optimal for our problem. Little previous work has looked into optimal noise adding mechanisms for approximate DP. \citet{Geng2016:Approximate} studied this problem to some extent, but did not explore truncated Laplace distributions. \citet{Steinke2018} discusses that truncated Laplace can be viewed as the canonical distribution for approximate DP as Laplace is canonical for pure DP. We propose to further explore the use of truncated Laplace distributions in approximate DP setting.



 \bibliographystyle{plainnat}
 \bibliography{./DataPrivacyBib}{}

\section{Detailed Proofs}
\begin{proof}[Proof of Theorem \ref{SuffStatThm}.]
  First note that $\phi_{\pi(\ul X)}$ satisfies $(\ep, \de)$-DP for all $\pi \in \sa(n)$, and that $\int \phi_{\pi(\ul X)} \ d\mu_i= \int\phi_{\ul X}\ d\mu_i$. Then
  \[\int\phi'_{\ul X}\ d\mu_i = \int \frac 1{n!} \sum_{\pi \in \sa(n)} \phi_{\pi(\ul X)}\ d\mu_i = \frac 1{n!} \sum_{\pi \in \sa(n)} \int \phi_{\pi(\ul X)}\ d\mu_i = \frac 1{n!} \sum_{\pi \in \sa(n)} \int\phi_{\ul X}\ d\mu_i = \int \phi_{\ul X}\ d\mu_i.\]
  To see that $\phi'$ satisfies $(\ep,\de)$-DP, we check condition \eqref{DPBinary}:
  \[\phi'_{\ul X} = \frac{1}{n!} \sum_{\pi \in \sa(n)} \phi_{\pi (\ul X)}
    \leq \frac{1}{n!} \sum_{\pi \in \sa(n)} (e^\ep \phi_{\pi(\ul X')} + \de)
    = \frac{1}{n!} \sum_{\pi \in \sa(n)} e^\ep \phi_{\pi(\ul X')} + \frac{1}{n!} \sum_{\pi \in \sa(n)} \de
    = \phi'_{\ul X'} + \de\]
  \begin{align*}
    (1- \phi'_{\ul X})=\l(1- \frac{1}{n!} \sum_{\pi \in \sa(n)} \phi_{\pi(\ul X)}\r)
    &= \frac{1}{n!} \sum_{\pi \in \sa(n)} (1- \phi_{\pi(\ul X)})\\
    &\leq \frac{1}{n!} \sum_{\pi \in \sa(n)} \l(e^\ep (1- \phi_{\pi(\ul X')}) + \de\r)
    =e^\ep \l(1- \phi'_{\pi (\ul X')}\r) + \de
    \end{align*}
\end{proof}
\begin{proof}[Proof of Lemma \ref{TulapLem}.]
  \begin{enumerate}
\item  We know that $L\overset d= G_1-G_2$, as shown in \citet{Inusah2006}. Let $f_U(\cdot)$ denote the pdf of $U$, and $F_U$ denote the cdf of $U$. We will use the property that $f_U(x)=f_U(-x)$ and $F_U(-x) =1-F_U(x)$. Then the pdf of $L+U$ is 
  \begin{align*}
    f_{L+U}(x) &= f_U(x-[x])\l(\frac{1-b}{1+b}\r) b^{|[x]|}= \begin{cases}
f_U(x-[x]) \l(\frac{1-b}{1+b}\r)b^{-[x]}&[x]\leq 0\\
f_U(x-[x])\l(\frac{1-b}{1+b}\r) b^{[x]}&[x]>0.
\end{cases}
  \end{align*}
If $[x]\leq 0$, then we have 
\begin{align*}
  F_{L+U}(x) &= \int_{-\infty}^x f_U(t-[t]) \l( \frac{1-b}{1+b}\r) b^{-[t]} \ dt\\
&= \int_{-\infty}^{[x]-1/2} f_U(t-[t])\l( \frac{1-b}{1+b}\r) b^{-[t]} \ dt + \int_{[x]-1/2}^xf_U(t-[x])\l( \frac{1-b}{1+b}\r) b^{-[x]} \ dt\\
&= \sum_{t=-\infty}^{[x]-1}\l(\frac{1-b}{1+b}\r) b^{-[t]} +  \int_{[x]-1/2}^xf_U(t-[x])\l( \frac{1-b}{1+b}\r) b^{-[x]} \ dt\\
&= \frac{b^{-[x]+1}}{1+b}+F_U(x-[x])\l(\frac{1-b}{1+b}\r) b^{-[x]}\\&= \frac{b^{-[x]}}{1+b}(b+F_U(x-[x])(1-b))
\end{align*}
Since, $L+U$ is symmetric about zero, as both $L$ and $U$ are symmetric about zero, we know that $F_{L+U}(-x) = 1-F_{L+U}(-x)$. The rest follows by replacing $x$ with $x-m$, and $F_U(x) = x+1/2$.
\item If $q=0$, then by part 1., it is clear that the output of Algorithm \ref{SampleTulap} has the correct distribution. If $q>0$, then by rejection sampling, we have that $N\sim \mathrm{Tulap}(m,b,q)$. For an introduction to rejection sampling, see \citet[Chapter 11]{Bishop2006}.
\item This property follows immediately from 1., and that $\mathrm{Tulap}(m,b,q)$ is truncated equally on both sides of $m$.
\end{enumerate}
\end{proof}

\begin{proof}[Proof of Lemma \ref{CLem}.]
The form of the cdf at integer values is easily verified from Lemma \ref{TulapLem}. 
It is clear that $C(m)$ is positive. It is also clear that $C(m)$ is continuous and monotone decreasing for all $m\in \RR\setminus \{z+1/2\mid z\in \ZZ\}$. So, we will check that $C$ is continuous at $m=z+1/2$ for $z\in \ZZ$. 
\[\lim_{\ep \downarrow 0} (1+b) C(z+1/2+\ep) = \lim_{\ep \downarrow 0} b^{z+1}(b+(1-\ep)(1-b)) = b^{z+1}\]
\[\lim_{\ep \uparrow 0} (1+b)C(z+1/2-\ep) = \lim_{\ep \uparrow 0} b^{z}(b+\ep(1-b)) = b^{z+1}\]
Since $C$ is continuous on $\RR$ and monotone decreasing almost everywhere, it follows that $C$ is monotone decreasing on $\RR$ as well.

Call $\al(m) = [m]-m+1/2$, which lies in $[0,1]$. Note that $\al(-m) = -[m]+m+1/2=1-\al(m)$. Then 
\begin{align*}
  (1+b)b^{-[m]}C(m)&=b+\al(m)(1-b)
=b+(1-\al(-m)(1-b))
=b+(1-b)-\al(-m)(1-b)\\
&=(b+1)-(b+\al(-m)(1-b))
=(b+1)(1-b^{[m]}C(m))\qedhere
\end{align*}
\end{proof}

\begin{proof}[Proof of Lemma \ref{RecurrenceLem1}.]
First we show that 1 and 2 are equivalent. Clearly the $m$ is the same for both. We must show that for $p\in (0,1)$, $e^\ep p\leq 1-e^{-\ep}(1-p)$ whenever $p\leq \frac{1}{1+e^{\ep}}$ and $e^{\ep}p>1-e^{-\ep}(1-p)$ when $p>\frac{1}{1+e^\ep}$. Setting equal $e^\ep p = 1-e^{-\ep}(1-p)$ we find that $p =\frac{1}{1+e^\ep}$. As $p\rightarrow 1$, we have that $e^\ep p > 1- e^{-\ep}(1-p)$ and as $p\rightarrow 0$, we have $e^\ep p <1-e^{-\ep}(1-p)$. We conclude that 1 and 2 are equivalent.

Next we show that 2 and 3 are equivalent. First we show that $F_{N_0}(x-m)$ satisfies the recurrence relation in 2. Set $b=e^{-\ep}$.
\[F_{N_0}([m]-1-m) = b^{-[m]+1}b^{[m]} \frac{(b+([m]-m+1/2)(1-b))}{1+b}\leq \frac{b}{1+b}=\frac{1}{1+e^\ep}\]
\[F_{N_0}([m]-m) = b^{-[m]}b^{[m]}\frac{(b+([m]-m+1/2)(1-b))}{1+b} \geq \frac{b}{1+b} = \frac{1}{1+e^\ep}.\]
Now, let $t\in \ZZ$ and check three cases:
\begin{itemize}
\item Let $t< [m]$, then $e^\ep F_{N_0}(t-m)=e^\ep b^{-t}C(m)=b^{-(t+1)}C(m) = F_{N_0}(t+1-m)$.
\item Let $t>=[m]$. Using Lemma \ref{CLem}, 
$1-e^{-\ep}(1-F_{N_0}(t-m)) 
= 1-b(1-b^{-[m]}C(m))
=1-b+b(1-b^{[m]}C(-m))
=1-b+b-b^{[m+1]}C(-m)
=F_{N_0}(t+1-m)$.
\item Let $t>m$. Then 
$1-e^{-\ep}(1-F_{N_0}(t-m))
=1-b(b^{t}C(-m))
=1-b^{t+1}C(-m)
=F_{N_0}(t+1-m)$.
\end{itemize}
Finally, for any value $c\in (0,1)$, we can find $m$ such that $F_{N_0}(0-m)=c$, by the intermediate value theorem. On the other hand, given $m$, set $\phi_0 = F_{N_0}(0-m)$.
\end{proof}

\begin{proof}[Proof of Lemma \ref{BingLem}.]
  Note that $(\phi_1 - \phi_2)(g-kf)\geq 0$ for almost all $x\in \mscr X$ (with respect to $\mu$). This implies that $\int (\phi_1 - \phi_2)(g-kf) \ d\mu \geq 0$. Hence, $\int \phi_1 g \ d\mu - \int \phi_2 g \ d\mu \geq k[\int \phi_1 f\ d\mu - \int \phi_2 f \ d\mu] \geq 0$.
\end{proof}

\begin{proof}[Proof of Theorem \ref{UMP1}.]
 First note that $\phi^*\in \mscr D_{\ep,0}^n$, since by Lemma \ref{RecurrenceLem1}, $\phi_x^* =\min\{e^\ep \phi^*_{x-1},1-e^{-\ep}(1-\phi^*_{x-1})\}$. So, $\phi^*$ satisfies \eqref{DP1}-\eqref{DP4}. Next, since by Lemma \ref{CLem}, $F_{N_0}(x-m)$ is a continuous, decreasing function in $m$ with $\lim_{m\uparrow \infty} F_{N_0}(x-m)=0$ and $\lim_{m\downarrow -\infty} F_{N_0}(x-m)=1$, we can find $m$ such that $E_{\ta_0} \phi^*_x=\al$ by the Intermediate Value Theorem. 

Now that we have argued that $\phi^*$ is a valid test, the rest of the result is an application of Lemma \ref{BingLem}. It remains to show that the assumptions are satisfied for the lemma to apply. Let $\phi\in \mscr D_{\ep,0}^n$ such that $E_{\ta_0} \phi_x\leq \al$. 

We claim that either $\phi_x=\phi^*_x$ for all $x\in \{0,1,2,\ldots, n\}$ or there exists $y$ such that $\phi_y>\phi^*$. To the contrary, suppose that $\phi^*_x\leq \phi_x$ for all $x$. Then either $\phi^*=\phi$ for all $x$, or there exists $z$ such that $\phi^*_z<\phi_z$. But, then $E_{\ta_0} \phi_x^*<E_{\ta_0} \phi_x\leq \al$ since the pmf of $\mathrm{Binom}(n,\ta)$ is nonzero for all $x\in \{0,1,2,\ldots,n \}$ contradicting the fact that $E_{\ta_0} \phi^*_x=\al$. We conclude that there exists $y$ such that $\phi^*_y >\phi_y$. 

Let $y$ be the smallest point in $\{0,1,2,\ldots, n\}$ such that $\phi^*_y>\phi_y$. We claim that for all $x\geq y$, we have $\phi^*_x\geq \phi_x$. We already know that for $y=x$, the claim holds. For induction, suppose the claim holds for some $x\geq y$. By Lemma \ref{RecurrenceLem1}, we know that $\phi^*_{x+1} = \min\{e^{\ep} \phi^*_x, 1-e^{-\ep} (1-\phi^*_x)\}$, and by constraints \eqref{DP1}-\eqref{DP4}, we know that $\phi_{x+1} \leq \min\{e^{\ep} \phi_x, 1-e^{-\ep}(1-\phi_{x})\}$. 
\begin{itemize}
\item Case 1: We have $\phi_x^*\leq \frac{1}{1+e^{\ep}}$. Then by Lemma \ref{RecurrenceLem1},
$\phi^*_{x+1} = e^\ep \phi^*_x \geq e^\ep \phi_x\geq \phi_{x+1}.$
\item Case 2: We have $\phi_x^*> \frac{1}{1+e^\ep}$. Then by Lemma \ref{RecurrenceLem1},
$\phi^*_{x+1} = 1-e^{-\ep}(1-\phi_x^*) \geq 1-e^{-\ep}(1-\phi_x)\geq \phi_{x+1}.$
\end{itemize}
 We conclude that $\phi^*_{x+1}\geq \phi_{x+1}$. By induction, the claim holds for all $x\in \{y,y+1,y+2,\ldots, n\}$. So, we have that $\phi^*_x\geq \phi_x$ for $x\in \{y,y+1,y+2,\ldots, n\}$ and $\phi^*_x\leq \phi_x$ for $x\in \{0,1,2,\ldots, y-1\}$. Since Binomial has a monotone likelihood ratio, by Lemma \ref{BingLem} we have that $E_{\ta_1} \phi^*_x \geq E_{\ta_1}\phi_x$. We conclude that $\phi^*$ is UMP-$\al$ among $\mscr D_{\ep,0}^n$ for the stated hypothesis test.
\end{proof}

\begin{proof}[Proof of Corollary \ref{OneSide1}.]
  Since $\phi^*_x$ and $\mathrm{Binom}(n,\ta)$ has a monotone likelihood ratio in $\ta$, $E_\ta \phi_x^*\leq E_{\ta_0}\phi_x=\al$ for all $\ta\leq \ta_0$ (property of MLR). So, $\phi^*$ is size $\al$. By Theorem \ref{UMP1}, we know that $\phi_x^*$ is most powerful for any alternative $\ta_1>\ta_0$ versus the null $\ta_0$. So, $\phi^*$ is UMP-$\al$. 

  Set $\ta_0'=1-\ta_0$ and pick $\ta_1<\ta_0$ and set $\ta'_1=1-\ta_1$. Finally set $X' = n-X$. We know given $X'\sim \mathrm{Binom}(n,\ta')$, the UMP-$\al$ among $\mscr D_{\ep,0}^n$ for the hypothesis $H_0: \ta'=\ta_0'$ versus $H_1: \ta'=\ta_1$ (note $\ta_1'>\ta_0'$) is $\phi^*_{x'} = F_{N_0}(x'-m)$ for some $m$ such that $E_{x' \sim \ta_0'}  \phi^*_x=\al$. We propose the rule $\psi^*_x = F_{N_0}(n-x-m) = 1-F_{N_0}(m+x-n)=1-F_{N_0}(x-m')$, where $m'$ is chosen such that $E_{x\sim \ta_0} \psi^*_x=\al$. The test $\psi^*$ is equivalent to the test $ \phi^*$, which we know is UMP-$\al$.
\end{proof}

\begin{proof}[Proof of Lemma \ref{RecurrenceLem2}.]
  We will abbreviate $F(x)\defeq F_{N_0}(x-m)$, where $N_0 \sim \mathrm{Tulap}(0,b=e^{-\ep},0)$ to simplify notation. First we will show that 1 and 2 are equivalent. It is clear that $y$ and $m$ are the same in both. Next consider 
$1-e^{-\ep}(1-p) + e^{-\ep}\de = e^\ep p+de$, solving for $p$ gives $p=\frac{1-\de}{1+e^{\ep}}$. Considering as $p\rightarrow 0$ and $p\rightarrow 1$, we see that $1-e^{-\ep}(1-p)+e^{-\ep} \de \geq e^\ep p + \de$ when $p \leq \frac{1-\de}{1+e^{\ep}}$ and $1-e^{-\ep}(1-p)+e^{-\ep} \de \leq e^\ep p + \de$ when $p \geq \frac{1-\de}{1+e^{\ep}}$. 

Next solving $1-e^{-\ep}(1-p)+e^{-\ep}\de=1$ for $p$ gives $p=1-\de$. So, $1-e^{-\ep}(1-p)+e^{-\ep}\de\leq 1$ when $p\leq 1-\de$ and $1-e^{-\ep}(1-p)+e^{-\ep}\de\geq$ when $p\geq 1-\de$. Lastly, solving $e^{\ep}p+\de=1$ for $p$ gives $p = \frac{1-\de}{e^{\ep}}\geq \frac{1-\de}{1+e^{\ep}}$. Combining all of these comparisons, we see that 1 is equivalent to 2.

Before we justify the equivalence of 2 and 3, we argue the following claim. Let $\phi_x$ be defined as in 3. Then $\phi_x\leq\frac{1-\de}{1+e^{\ep}}$ if and only if $F(x)\leq\frac{1}{1+e^{\ep}}$. Suppose that $\phi_x \leq \frac{1-\de}{1+e^{\ep}}$. Then $\frac{F(x)-q/2}{1-q} \leq \frac{1-\de}{1+e^{\ep}}$. Then we have that
\begin{align*}
F(x)&\leq \frac{(1-q)(1-\de)}{1+e^{\ep}} + \frac q2\\
 &=\frac{1}{1+e^{\ep}}\l[(1-q)(1-\de) + \l(\frac{b+1}{b}\r) \frac{q}{2}\r]\\
 &=\frac{1}{1+e^{\ep}} \l[\frac{(1-b)(1-\de)}{1-b+2\de b} + \l(\frac{b+1}{b}\r)\frac{\de b}{1-b+2\de b}\r]\\
 &=\frac{1}{1+e^{\ep}} \l[1-b+2 \de b]^{-1}[(1-b)(1-\de) + (b+1)\de\r]\\
 &=\frac{1}{1+e^{\ep}}
 \end{align*}
We are now ready to show that $\phi_x$ as described in 3 fits the form of 2. 
\begin{itemize}
\item Suppose that $0<\phi_x<\frac{1-\de}{1+e^{\ep}}$. By the above, we know that $F(x)\leq \frac{1}{1+e^{\ep}}$. Then by Lemma \ref{RecurrenceLem1}
\[e^\ep \phi_x + \de = \frac{e^{\ep} F(x) - \frac{q}{2b}}{1-q} + \de
=\frac{F(x+1) - \frac{q}{2}}{1-q} + \frac{\frac{q}{2}-\frac{q}{2b}}{1-q}+\de
=\phi_{x+1} + \frac{\de b}{1-b}\l(1-\frac{1}{b}\r) + \de
=\phi_{x+1}\]
\item Suppose that $\frac{1-\de}{1+e^{\ep}}<\phi_x \leq 1-\de$. Then we have $F(x) >\frac{1}{1+e^{\ep}}$. Then 
  \begin{align*}
    1-e^{-\ep}(1-\phi_x) + e^{-\ep} \de
&=1-e^{-\ep}\l(1-\frac{F(x) - q/2}{1-q}\r) + e^{-\ep} \de\\
&= (1-q)^{-1}[1-e^{-\ep}(1-F(x)) + bq/2-q] + b\de\\
&= (1-q)^{-1}[F(x+1) - q/2] + \frac{(b-1)q/2}{1-q} + b\de\\
&= \phi_{x+1} + \frac{\de b(b-1)}{1-b} + b\de\\
&= \phi_{x+1}
  \end{align*}
\item Finally, we must show that if $\phi(x)=1$ then $\phi(x-1)\geq 1-\de$. It suffices to show that $F(x)\geq 1-q/2$ implies that $F(x-1) \geq (1-\de)(1-q)+q/2=1-(1/b)(q/2)$. We prove the contrapositive. Suppose that $F(x-1)<1-(1/b)(q/2)$. Then since $F$ satisfies property \eqref{DP3}, we know that
  \begin{align*}
    F(x) \leq 1-e^{-\ep}(1-F(x-1))&< 1-b(1-(1-(1/b)(q/2)))\\
    &= 1-b(1-1+(1/b)(q/2)) = 1-q/2.\end{align*}
    \end{itemize}

We have justified that $\phi_x$ in 3 satisfies the recurrence relation in 2. Given $\phi'$ according to 2 with first non-zero entry at $y$, by Lemma \ref{CLem} and Intermediate Value Theorem, we can find $m\in \RR$ such that $\phi_y=\phi'_y$. We conclude that 1,2, and 3 are all equivalent.
\end{proof}

\begin{proof}[Proof of Lemma \ref{CDFLem}.]
  Let $\al \in [0,1]$ be given. We will only consider $B\subset \RR$ (Lebesgue measurable) such that $\mu_{X'}(B)=\al$. Then demonstrating $(\ep,\de)$-DP requires $\sup\limits_{\{B\mid \mu_{X'}(B)=\al\}} \mu_X(B) \leq e^{\ep} \al + \de$. We interpret this problem as testing the hypothesis $H_0: X'$ versus $H_1: X$, using the rejection region $B$, where $\al$ is the type I error, and $\mu_X(B)$ is the power. We know that $\sup\limits_{\{B\mid \mu_{X'}(B)=\al\}} \mu_X(B)$ is achieved by the Neyman-Pearson Lemma. Since $\mu_X$ has an MLR in $T(X)$, $\arg\sup_{B\mid \mu_{X'}(B)=\al} \mu_X(B)$ is of the form of either $(-\infty,t)$ or $(t,\infty)$ depending on whether $T(X)$ is greater or lesser than $T(X')$. Note that since $\mu_X$ is dominated by Lebesgue measure for all $X$, $\mu_X((-\infty,t))$ is continuous in $t$, which allows us to achieve exactly $\al$ type I error.
\end{proof}

\begin{proof}[Proof of Theorem \ref{TulapDP}.]
Let $Z| X \sim \mathrm{Tulap}\l(T(X),b=e^{-\ep},\frac{2\de b}{1-b+2\de b}\r)$. We know that the distribution of $Z$ is symmetric with location $T(X)$, and the pdf $f_Z(t)$ is increasing as a function of $|t-T(X)|$. It follows that $f_Z(t)$ has a MLR in $T(X)$. By Lemma \ref{RecurrenceLem2}, we know that $\phi_x=F_Z(m)$ satisfies \eqref{DP1}-\eqref{DP4}, so by Lemma \ref{CDFLem}, we have the desired result. 
\end{proof}

\begin{proof}[Proof of Theorem \ref{PValueResult}.]
  Recall that a $p$-value for $H_0: \ta\in \Ta_0$ versus $H_1: \ta\not \in \Ta_0$, given $X\sim f_\ta$ is a real-valued random variable $p(X)\in [0,1]$ such that for every $\ta\in \Ta_0$ and every $0<\al<1$, $P_{X\sim f_\ta}(p(X) \leq \al)\leq \al$.
  \begin{enumerate}
  \item First we show that $p(\ta_0,Z)$ is a $P$-value. We consider 
    \begin{align*}
      \sup_{\ta\leq \ta_0} P_{\substack{Z| X \sim \mathrm{Tulap}(X,b,q)\\X\sim \mathrm{Binom}(n,\ta)}}(p(\ta,Z) \leq \al)
&= P_{\substack{Z| X \sim \mathrm{Tulap}(X,b,q)\\X\sim \mathrm{Binom}(n,\ta_0)}}(p(\ta_0,Z)\leq \al)\\
    \end{align*}
using the fact that $X$ has a monotone likelihood ratio in $\ta$. Note that $p(\ta_1,Z) = F_{-(X+N)}(-Z)$ where $N\sim \mathrm{Tulap}(0,b,q)$ and $X\sim \mathrm{Binom}(n,\ta_0)$. But, then when $\ta=\ta_0$, $-Z \overset d = -(X+N)$. So, $p(\ta_0,Z) = F_{-Z}(-Z) \sim \mathrm{Unif}(0,1)$. So, 
\[P_{\substack{Z| X \sim \mathrm{Tulap}(X,b,q)\\X\sim \mathrm{Binom}(n,\ta_0)}}(p(\ta_0,Z)\leq \al)
=P_{U\sim \mathrm{Unif}(0,1)}( U\leq \al)=\al\]
\item Let $N\sim \mathrm{Tulap}(0,b,q)$, and recall from Theorem \ref{UMP2} that the UMP-$\al$ test for $H_0: \ta\leq \ta_0$ versus $H_1: \ta>\ta_0$ is $\phi_x^* = F_N(x-m)$, where $m$ satisfies $E_{\ta_0}\phi^*_x =\al$. We can write $\phi^*$ as 
  \begin{align*}
    \phi^*_x = F_N(x-m)
&= P_{N\sim \mathrm{Tulap}(0,b,q)}(N\leq X-m\mid X)\\
&= P_N(X+N\geq m\mid X)
= P_{Z| X \sim \mathrm{Tulap}(X,b,q)} (Z\geq m\mid X)
  \end{align*}
where $m$ is chosen such that 
\begin{align*}
\al = E_{\ta_0} \phi^*_x
 &= E_{x\sim \ta_0} P_{Z| X \sim \mathrm{Tulap}(X,b,q)}(Z\geq m \mid X) \\
&= P_{\substack{Z| X\sim \mathrm{Tulap}(X,b,q)\\X\sim \mathrm{Binom}(n,\ta_0)}} (Z\geq m)
=1-F_Z(m),
\end{align*}
where $F$ is the cdf of the marginal distribution of $Z$, where $Z| X \sim \mathrm{Tulap}(X,b,q)$ and $X\sim \mathrm{Binom}(n,\ta_0)$. From this equation, we have that $m$ is the $(1-\al)$-quantile of the marginal distribution of $Z$. 

Let $R| X \sim \mathrm{Bern}(\phi^*_X)$ and $Z|X\sim \mathrm{Tulap}(X,b,q)$. Then 
\begin{align*}
  R| X\overset d = I(Z\geq m) \mid X&\overset d= I\l(1-\al \leq F_{\substack{X' \sim \mathrm{Binom}(n,\ta_0)\\ Z'| {X'}\sim \mathrm{Tulap}(n,b,q)}} (Z)\r) \Big | {X}\\
&\overset d= I\l( P_{\substack{X'\sim \mathrm{Binom}(n,\ta_0)\\ N\sim \mathrm{Tulap}(0,b,q)}}(X'+N \geq Z\mid Z)\leq \al\r)\Big|X\\
\end{align*}
Taking the expected value over $R| X$ and $Z| X$ of both sides gives
\[\phi^*_X = E(R\mid X) = P_{Z| X\sim \mathrm{Tulap}(X,b,q)}(p(\ta_0,Z)\leq \al\mid X).\]
\item We can express $p(\ta_0,Z)$ in the following way:
  \begin{align*}
    p(\ta_0,Z)&= P_{\substack{X\sim \mathrm{Binom}(n, \ta_0)\\ N\sim \mathrm{Tulap}(0,b,q)}} (X+N \geq Z)
= P_{X,N}(-N\leq X-Z)\\
&= E_{X\sim \mathrm{Binom}(n, \ta_0)} P_{N}(N\leq X-Z\mid X)
= E_{X\sim \mathrm{Binom}(n,\ta_0)}F_N(X-Z)\\
&= \sum_{x=0}^n F_N(x-Z)\binom nx \ta_0^x(1-\ta_0)^{n-x},
  \end{align*}
which is just the inner product of the vectors $\ul F$ and $\ul B$ in algorithm \ref{PValueAlgorithm}.\qedhere
  \end{enumerate}
\end{proof}

\end{document}